\documentclass{amsart}

\usepackage[margin=1.0in]{geometry}

\usepackage[utf8]{inputenc}
\usepackage{amssymb,amscd,amsthm, verbatim,amsmath,color,fancyhdr, mathrsfs}
\usepackage{graphicx}
\usepackage{xfrac}
\usepackage{tikz, float, tikzscale}
\usepackage{bbm}
\usepackage[hidelinks]{hyperref}
\usepackage{physics}
\usepackage{breqn}
\usepackage{amsmath}

\usepackage{mathtools}
\usepackage{pythonhighlight} %For typesetting Python

\newtheorem{thm}{Theorem}
\newtheorem*{thm*}{Theorem}
\newtheorem*{lem*}{Lemma}

\newtheorem{lem}[thm]{Lemma}
\newtheorem{rem}[thm]{Remark}
\newtheorem*{rem*}{Remark}
\newtheorem{cor}[thm]{Corollary}
\newtheorem*{cor*}{Corollary}

\newtheorem{problem}[thm]{Problem}

\newtheorem*{bur}{Burgess's bound for short character sums}

\begin{document}

\title{Quantitative Khintchine on the Parabola with Non-monotonic Approximation Functions} %--------------Title???----------------

\author{Maiken Gravgaard}
\address{Maiken Gravgaard, Department of
	Mathematics, Ian Wand Building, University of York, YO10 5GH York, United Kingdom}
\email{maiken@balmangravgaard.dk} 
\author{Simon Kristensen}
\address{Simon Kristensen, Department of Mathematics, Aarhus
	University, Ny Munkegade 118, 8000 Aarhus C, Denmark.}
\email{sik@math.au.dk}
\thanks{Funded by Aarhus University Research
	Foundation, grant no. AUFF-E-2021-9-20}
	
\begin{abstract} %--------ABSTRACT--------
  We prove a quantitative version of the convergence case of
  Khintchine's celebrated theorem in metric Diophantine approximation,
  but where the approximated points are restricted to lying on the
  parabola. A novel feature of our result is that unlike other results
  in literature, the approximating function is no longer required to
  be monotonic. This requires us to obtain explicit constants in
  classical number theoretic results, most notably in Burgess' bound
  for character sums in short intervals.
\end{abstract}

\maketitle	

\section{Introduction}\label{SectionIntroduction} %--------INTRODUCTION--------

Begin by recalling the notion of \emph{\(\psi\)-approximable
points}. Let $\psi$ be a function $\psi:\mathbb{N}\to [0,\infty)$,
  which we will hereby refer to as the \emph{approximation
  function}. Let $n\in\mathbb{N}$, then we define the set of
  \emph{simultaneously \(\psi\)-approximable points} in $\mathbb{R}^n$
  as follows
\begin{equation*}
  \mathcal{S}_n(\psi) :=\left\{(x_1,\dots, x_n)\in \mathbb{R}^n \mid
  \exists_{\infty} (p_1,\dots, p_n, q)\in\mathbb{Z}^n\times\mathbb{N}
  \text{ s.t. } \left\lvert x_i - \frac{p_i}{q} \right\rvert <
  \frac{\psi(q)}{q} ~ , 1\leq i\leq n\right\}.
\end{equation*}
By Dirichlet's Theorem \cite{Cassels1965}, $\mathcal{S}_{n}(q\mapsto
q^{-\frac1n}) = \mathbb{R}^n$, and for functions decaying faster than
$q\mapsto q^{-\frac1n}$ Khintchine proved his famous 0--1-law, which
we state a version of below:
\begin{thm*}[Khintchine, 1924]
  $\lambda_n$--almost no points are simultaneously $\psi$--well
  approximable if the series $\sum_{q=1}^{\infty} \psi(q)^n$
  converges.  If we additionally assume, that the approximation
  function $\psi$ is monotonic, then $\lambda_n$-almost every point is
  simultaneously $\psi$--well approximable, provided the series
  $\sum_{q=1}^{\infty} \psi(q)^n$ diverges.
\end{thm*} 
Here, $\lambda_n$ is the $n$-dimensional Lebesgue measure.  Note here,
that we do not assume monotonicity of the approximation function in
the convergence case! (It is however both assumed and indeed needed in
the divergence case, see \cite{DuffinSchaeffer1941}).

The present work deals with Diophantine approximation on
manifolds. Here, one restricts the point $\mathbf{x}$ to be on some
manifold, which is itself a Lebesgue null set with respect to the
Lebesgue measure on the ambient space, but which nonetheless carries a
natural measure. In this case, geometric conditions on the manifold
play an important role. For instance, if the manifold in question is
a coordinate hyperplane, one would only be observing Diophantine
phenomena from the lower dimensional space. The usual way of avoiding
this is to assume that the manifold 'curves out of every affine
subspace', a notion which can be formalised as an explicit
non-degeneracy condition.

In \cite{BeresneveichYang2023} Beresnevich and Yang show the
convergence case for non-degenerate submanifolds $\mathcal{M}\subset
\mathbb{R}^n$ for monotonic $\psi$, and it was recently shown by
Beresnevich, Datta and Yang \cite{arXiv:2602.11045} that a full (and
even more general) version of Khintchine's theorem holds for such
manifolds. However, it is critical for these results that the
approximation function is monotonic, even in the convergence case.

As pointed out in \cite{Adiceametal2016}, an immediate consequence of
Beresnevich and Yang's theorem is
\begin{cor}[Corollary 1 in \cite{Adiceametal2016}]\label{Corollary}
  Let $\psi$ be monotonic such that $\sum_{q=1}^{\infty} \psi(q)^n$
  converges. Then for almost every (with respect to the natural
  measure) $\mathbf{x}\in\mathcal{M}$, there exists a
  $\kappa(\mathbf{x})>0$ such that
  \begin{equation*}
    \lVert q\mathbf{x} \rVert > \kappa(\mathbf{x})\psi(q),
  \end{equation*}
  for all $q\in\mathbb{N}$.
\end{cor}

Applications of this often require that $\kappa(\mathbf{x})$ is
independent of $\mathbf{x}$, which is impossible to guarantee on a set
of full Lebesgue measure. In \cite{Adiceametal2016} Adiceam,
Beresnevich, Levesley, Velani and Zorin state the following problem
\begin{problem}[Problem in \cite{Adiceametal2016}]\label{Problem}
  Investigate the dependency between $\kappa>0$ and the probability of
  the set of vectors $\mathbf{x}\in\mathcal{M}$ where
  $\kappa(\mathbf{x})=\kappa$.
\end{problem}
This problem was addressed by Datta in \cite{Datta2024}, who found a lower estimate on the measure of the set of points in
$\mathcal{M}$ satisfying the conditions of Corollary \ref{Corollary},
but with $\kappa$ being a fixed, small number. Again, for her proof to
work, it is imperative that the approximation function $\psi$ is
monotonic.

Removing monotonicity from Khintchine's theorem is in general a hard
problem, and in fact it is strictly necessary for $n=1$ where a
different volume sum is needed in the non-monotonic case, see the
celebrated paper of Koukoulopoulos and Maynard
\cite{zbMATH07239275}. For manifolds, even removing it in the case of
convergence is hard and often requires the manifold to be very special
indeed. To the best of our knowledge, the only result of its kind is
the following, due to Huang \cite{Huang2020}.

\begin{thm}[Theorem 1 in \cite{Huang2020}]\label{Huangthm1}
  Let $\psi$ be \emph{any} approximation function, such that
  $\sum_{q=1}^{\infty} \psi(q)^2$ converges. Then almost no points on
  the parabola $\mathcal{P}:= \{(x,x^2)\mid x\in[0,1]\}$ are
  simultaneously $\psi$--well approximable.
\end{thm}

As in \cite{Adiceametal2016} and \cite{Datta2024}, one can formulate
statements analogous to Corollary \ref{Corollary} and Problem
\ref{Problem}, which is exactly what we aim to address.

Define
\begin{equation}
  \mathcal{B}'(\psi, \kappa) := \left\{x\in[0,1] : \max\left(\lVert
  qx\rVert, \lVert q x^2\rVert\right)> \kappa \psi(q) \text{ for all }
  q\in\mathbb{N}\right\}
\end{equation}
We give two versions of a quantitative Khintchine's Theorem on the
parabola with non-monotonic approximation functions for general
composite $q$. The first, Theorem \ref{thmallq}, holds for all $q$, at
the cost of the constants appearing in $\kappa(\delta)$ being
\emph{very} small!  The other version, Theorem \ref{thmlargeq}, has
more reasonable constants, but this comes at the cost of only
considering \emph{very} large $q$ with large prime divisors.

Additionally we give two versions, where we assume that $q$ does not
have too many divisors (namely $q=p_1\cdots p_n$ for a fixed $n$ in
Theorem \ref{thmp1...pn} and $q=p$ in Theorem \ref{thmprime}).

\begin{thm}[All $q$]\label{thmallq}
  Let $\psi:\mathbb{N}\to [0, \infty)$ be any function such that
    $S_{\psi}:=\sum_{q=1}^{\infty} \psi(q)^2$ converges. Given any
    $0<\delta<1$, let
    
    \begin{align*}
      \kappa(\delta) < \min
      &\left\{\frac{1}{6\max_{n\in\mathbb{N}}\psi(n)} ,
      \frac{\delta}{\sqrt{S_{\psi}}} \frac{2 \sqrt{6}
        (5\log(2))^4}{3\pi^3 \sqrt{\zeta\left(\frac{12}{11}\right)}
        11^4}, \sqrt{\frac{\delta}{S_{\psi}}} \frac{1}{\sqrt{18}\pi},
      \right.\nonumber\\ &\left. {}
      \left(\frac{\delta}{S_{\psi}}\right)^{\frac23}
      \left(\frac{1}{2\pi^2 \sqrt{12}} \right)^{\frac23}\Bigg(
      99684\cdot 4.19253564643679 \right.\nonumber\\ & \left. {}
      +27\cdot 4^{\frac{3}{16} + \frac{0.11237+0.11237^2}{2}}
      3^{\frac52} \left(e^{e^8}-10^5\right) e^{\frac{5\log(2)
          e^8}{2(8-1.39177)}}\left(\frac{1}{e\cdot
        0.11237}\right)^{1+\frac{0.11237}{2}} \right.\nonumber\\ &
      \left. {} +5\cdot
      4^{\frac{3}{16}+\frac{0.00218819+0.00218819^2}{2} +
        \frac{5}{2\cdot 57}}
      \left(\frac{57}{\log(2)}\right)^{\frac{7\cdot
          176846309399143769411680}{2}} \left(\frac{1}{e\cdot
        0.00218819}\right)^{1+\frac{0.00218819}{2}}\Bigg)^{\frac{-2}{3}}
      \right\} \\ \simeq \min & \left\{
      \frac{1}{6\max_{n\in\mathbb{N}}\psi(n)}
      ,\frac{\delta}{\sqrt{S_{\psi}}} 0.00015248265228 ,
      \sqrt{\frac{\delta}{S_{\psi}}} 0.07502635967975,
      \left(\frac{\delta}{S_{\psi}}\right)^{\frac23}
      10^{-10^{23.89775276641198}}\right\}.
    \end{align*}
    
    Then 
    \begin{equation*}
      \lambda\left(\mathcal{B}'(\psi, \kappa)\right) \geq 1-\delta.
    \end{equation*}
\end{thm}

\begin{thm}[Large $q$]\label{thmlargeq}
  Let $\psi:\mathbb{N}\to [0, \infty)$ be any function such that
    $S_{\psi}:=\sum_{q=1}^{\infty} \psi(q)^2$ converges. Assume
    additionally, that $\psi(q)=0$ for all $q\leq
    \exp(\exp(41))$. Given any $\delta > 0$, let
    \begin{align*}
      \kappa(\delta) < \min&\left\{
      \frac{1}{6\max_{n\in\mathbb{N}}\psi(n)}, \frac{\delta}{
        \sqrt{S_{\psi}}} \frac{2}{3\pi^2
        \sqrt{\frac{\pi^2}{6}\zeta\left(\frac{112}{57}\right)}},
      \sqrt{\frac{\delta}{S_{\psi}}} \sqrt{\frac{1}{18\pi^2}}
      \right. \\ & \left. {}
      \left(\frac{\delta}{S_{\psi}}\right)^{\frac23}
      \left(2\pi^2\sqrt{12}\left(4.19253564643679+ 27\cdot
      4^{\frac{3}{16}+\frac{5}{2\cdot 57} +
        \frac{\frac{41}{e^{41}}+\left(\frac{41}{e^{41}}\right)^2}{2}}
      \right)\right)^{\frac{-2}{3}}\right\} \\ \simeq \min&\left\{
      \frac{1}{6\max_{n\in\mathbb{N}}\psi(n)},
      \frac{\delta}{\sqrt{S_{\psi}}} 0.04064400,
      \sqrt{\frac{\delta}{S_{\psi}}} 0.07502635,
      \left(\frac{\delta}{S_{\psi}}\right)^{\frac23}0.00499238\right\}.
	\end{align*}
    Then
    \begin{equation*}
      \lambda\left(\mathcal{B}'(\psi,\kappa)\right)\geq 1-\delta.
    \end{equation*}
\end{thm}

\begin{thm}[$q=p_1\dots p_k$]\label{thmp1...pn}
  Let $\psi:\mathbb{N}\to [0, \infty)$ be any function such that
    $S_{\psi}:=\sum_{q=1}^{\infty} \psi(q)^2$ converges. Let
    $k\in\mathbb{N}$ and assume additionally, that $\psi(q)=0$ for all
    $q$ with more than $k$ prime divisors. Given any $0<\delta < 1$,
    let
    \begin{align*}
      \kappa(\delta) <
      \min&\left\{\frac{1}{6\max_{n\in\mathbb{N}}\psi(n)},
      \frac{\delta}{\sqrt{S_{\psi}}}\frac{4}{3\pi^4}\frac{1}{2^n} ,
      \left(\frac{\delta}{S_{\psi}}\right)^{\frac{2}{3}}\left(2\pi^2\sqrt{12}\Bigg[4.19253564643679\cdot
        99684 \right.\right. \\ &\quad \left.\left. {} + 27\cdot
        4^{\frac{3}{16}+\frac{0.11237+0.11237^2}{2}}3^{\frac52}\left(\frac{1}{e\cdot
          0.11237}\right)^{1+\frac{0.11237}{2}} 2^{\frac{7k}{2}}
        \Bigg]\right)^{\frac{-2}{3}}, \sqrt{\frac{\delta}{S_{\psi}}}
      \sqrt{\frac{1}{18\pi^2}} \right\} \\ \simeq
      \min&\left\{\frac{1}{6\max_{n\in\mathbb{N}}\psi(n)},\frac{\delta}{\sqrt{S_{\psi}}}
      0.0410639\frac{1}{2^n}, \right. \\ &\quad \left. {}
      \left(\frac{\delta}{S_{\psi}}\right)^{\frac23} 0.0598025
      \left(417928.72 + 2082.9292 \cdot
      2^{\frac{7k}{2}}\right)^{\frac{-2}{3}},
      \sqrt{\frac{\delta}{S_{\psi}}} 0.0750263\right\}.
    \end{align*}
    Then
    \begin{equation*}
      \lambda\left(\mathcal{B}'(\psi, \kappa)\right) \geq 1-\delta.
    \end{equation*}
\end{thm}

\begin{thm}[$q=p$ prime]\label{thmprime}
  Let $\psi:\mathbb{N}\to [0, \infty)$ be any function such that
    $S_{\psi}:=\sum_{q=1}^{\infty} \psi(q)^2$ converges. Assume
    additionally, that $\psi(q)=0$ whenever $q$ is not a prime. Given
    any $0<\delta < 1$, let
    \begin{align*}
      \kappa(\delta) &< \min\left(\frac{1}{6\max_{p}\psi(p)} ,
      \frac{2\delta}{3\pi^2 S_{\psi}}, \left(\frac{\delta (e\cdot
        0.249)^{\frac14}}{2 \pi^2\sqrt{12}\cdot 10.0366
        S_{\psi}}\right)^{\frac23},\sqrt{\frac{\delta}{18\pi^2
          S_{\psi}}} \right)\\ &\simeq \min\left(
      \frac{1}{6\max_{p}\psi(p)} , \frac{\delta}{S_{\psi}}
      0.0675474558, \frac{\delta^{\frac23}}{S_{\psi}^{\frac23}}
      0.0120432671,\sqrt{\frac{\delta}{S_{\psi}}} 0.0750263597
      \right).
    \end{align*}
	
    Then
    \begin{equation*}
      \lambda\left(\mathcal{B}'(\psi, \kappa)\right) \geq 1-\delta.
    \end{equation*}
\end{thm}

Note that Theorem \ref{thmp1...pn} of course applies here as well with
$k=1$, but the bounds are not easily comparable. It is a consequence
of the versions of Burgess' theorem used in the different proofs.  We
could of course put the bounds from Theorem \ref{thmp1...pn} into the
minimum of Theorem \ref{thmprime}, but we have opted for the above
version which is already complicated to state.

All four proofs follow the same structure. For this reason, we will
first prove Theorem \ref{thmallq} in detail, after which Theorems
\ref{thmlargeq}, \ref{thmp1...pn}, and \ref{thmprime} are shown with
references to the main proof.  The structure of the proof is as
follows: Reduce the problem to counting rational points with fixed
denominator near the parabola. Then this counting function is
estimated by exponential sums, which in turn are rewritten into sums
of Dirichlet characters. Up to this point we have followed the same
steps as \cite{Huang2020}, making sure to keep all constants and lower
order terms along the way. As in \cite{Huang2020} we want to estimate
our character sums using Burgess's bound for short sums of Dirichlet
characters, the problem is, that there is no effective version that
holds for all composite moduli $q$. In Section \ref{SectionBurgess} we
take care of this problem.

\section{Notes on Explicit Burgess bounds}\label{SectionBurgess} %--------EXPLICIT BURGESS--------

Before we get into the proofs of Theorems \ref{thmallq} and
\ref{thmlargeq} we need an explicit version of Burgess's famous bound
for short sums of Dirichlet characters \cite{Burgess1963}.
\begin{bur}
  Let $\chi$ be a non-principal character modulo $q$. Then for
  any $\varepsilon>0$ we have
  \begin{equation*}
    \left\rvert \sum_{M < n\leq M+N} \chi(n) \right\rvert
    \ll_{\varepsilon} \sqrt{N} q^{\frac{11}{16}+\varepsilon},
  \end{equation*}
  where the implied constant only depends on $\varepsilon$.
\end{bur}
We will need an explicit version, where the implied constant is
known. In addition, we will need versions both for composite and prime
moduli.

For composite moduli $q$ we split into three cases. For large moduli,
we refer to a result of Bordignon \cite{Bordignon2022}(See also
\cite{Jain-Sharmaetal2021}).

\begin{lem}[Large $q$. Theorem 1.3 in \cite{Bordignon2022}]\label{burgesslargeq}
  Let $q > \exp(\exp(8))$. Let $\chi$ be a primitive Dirichlet
  character mod $q$. For any integers $M$ and $N<q$ we have
  \begin{equation*}
    \left\lvert \sum_{M<n\leq M+N} \chi(n) \right\rvert \leq 5
    d(q)^{\frac32}\sqrt{N}q^{\frac{3}{16}}\sqrt{\log(q)\log(\log(q))}
  \end{equation*}
\end{lem}

Here and throughout, $d(q)$ denotes the number of divisors of $q$.

For small values of $q$, the constant may be calculated by brute
force, which we have done using \texttt{Python}.

\begin{lem}[Small $q$. See Appendix \ref{SectionPython}]\label{burgesssmallq}
  For $1\leq q \leq 10^5$ and $q^{\frac{3}{8}}\leq N\leq q^{\frac58}$ we have
  \begin{equation*}
    \left\lvert\sum_{j=1}^{N} \chi(j) \right\rvert \leq
    4.19253564643679\sqrt{N}q^{\frac{3}{16}},
  \end{equation*}
  where $\chi(j) = \left(\frac{j}{q}\right)$ or $\chi(j) =
  \left(\frac{q}{j}\right)$.
\end{lem}

This leaves the intermediate range. For this, we use again results of
Bordignon \cite{Bordignon2022}, but in this case additional work is
needed.

\begin{lem}[Medium $q$]\label{burgessmediumq}
  Let $10^5 \leq q \leq \exp(\exp(8))$ and let $\chi$ be a primitive
  Dirichlet character mod $q$. For any integers $M$ and $N<q$ we have
  \begin{equation*}
    \left\lvert \sum_{M<n\leq M+N} \chi(n) \right\rvert \leq 27
    d(q)^{\frac32}\sqrt{N}q^{\frac{3}{16}}\sqrt{\log(q)\log(\log(q))}
  \end{equation*}
\end{lem}

\begin{proof}
  The statement follows from Theorem 3.4 in \cite{Bordignon2022},
  which we state a simplified version of. Below we take $\tau(q)$ to
  be the number of prime divisors of $q$, $C$ to be the
  Euler--Mascheroni constant, and $\varphi(q)$ to be Euler's totient
  function.
	
  Let $q$ be an integer. Let $g\geq2$ be a real number and let $m$ be
  a positive real number. Let $\chi$ be a primitive Dirichlet
  character modulo $q$. Assume that
  \begin{equation}
    q > \max\left\{ \left( \max\left\{29,
    2^{\tau(q)+1}\frac{q}{\varphi(q)}+1\right\}\frac{g}{m^2 \log(q)
      \log(\log(q))}\right)^{8}, \left(\frac{12}{g}\right)^{4}, h^4,
    \left(\frac{16\log(q)}{m^2 \log(\log(q))}\right)^8
    \right\}. \label{BordignonThm3.4}
  \end{equation}
  Define
  \begin{equation*}
    v_1(m,q) = \frac{2\left(1+\frac{2}{e\log(q)}\right)}{m},
  \end{equation*}
  \begin{equation*}
    v_2(m,q,g) = \frac{2 v_1(m,q)^4}{g} + \frac{2\log(\log(q))^2
      \log\left(1.85 v_1(m,q)^2
      q^{\frac38}\frac{\log(q)}{g\log(\log(q))}\right)}{\log(q)^2},
  \end{equation*}
  and
  \begin{equation*}
    v_3(m,q,g,h) = 2g\left( 1-\frac{1}{h} - \frac{g}{m^2
      q^{\frac18}\log(q)\log(\log(q))}\right)^{-1}\left( \frac{17
      v_2(m,q,g)}{4g^3}\right)^{\frac14}\left(e^C+\frac{2.51}{\log(\log(q))^2}\right)
    + \frac{2}{\sqrt{g}}.
  \end{equation*}
  If 
  \begin{equation}
    v_3(m,q,g,h)\leq m \label{vleqm}
  \end{equation}
  holds, then for any integers $M, N$, we have
  \begin{equation*}
    \left\lvert \sum_{M<n\leq M+N} \chi(n) \right\rvert \leq m
    d(q)^{\frac32}\sqrt{N}q^{\frac{3}{16}}\sqrt{\log(q)\log(\log(q))}
  \end{equation*}
	
  We claim, that for $g=2$, $h=17$, $m=27$ and $10^5\leq q\leq
  \exp(\exp(8))$, equations (\ref{BordignonThm3.4}) and (\ref{vleqm})
  hold.  Indeed
  \begin{align*}
    \left( 29\frac{2}{27^2 \log(q)\log(\log(q))}\right)^8 \frac{1}{q}
    \leq \left( 29\frac{2}{27^2 \log(10^5)\log(\log(10^5))}\right)^8
    \frac{1}{10^5} <1.
  \end{align*}

  We now state two lemmas from literature giving bounds on arithmetic
  functions. The first is a lower bound for the Euler function.
  \begin{lem*}[Theorem 8.8.7 in \cite{BachShallit1996}]
    \begin{equation*}
      \varphi(q)> \frac{q}{e^C\log(\log(q)) + \frac{3}{\log(\log(q))}}.
    \end{equation*}
  \end{lem*}
  We also need an upper bound for the $\tau$-function.
  \begin{lem*}[Theorem 12 in \cite{Robin1982}]
    \begin{equation*}
      \tau(q)\leq \frac{\log(q)}{\log(\log(q))}
      \left(1+\frac{1.4573}{\log(\log(q))}\right).
    \end{equation*}
  \end{lem*}
  
  Using these, we obtain
  \begin{align*}
    &\left[\left( 2^{\tau(q)+1}\frac{q}{\psi(q)} +1\right)
      \frac{2}{m^2\log(q )\log(\log(q))}\right]^{8}\frac{1}{q}
    \\ &\leq \left[\left(
      2^{\frac{\log(q)}{\log(\log(q))}\left(1+\frac{1.4573}{\log(\log(q))}\right)+1}\left(1.7811\log(\log(q))
      + \frac{3}{\log(\log(q))}\right)+1\right) \frac{2}{m^2\log(q
        )\log(\log(q))}\right]^{8}\frac{1}{q} \\ &\leq \left[\left(
      2^{\frac{\log(10^5)}{\log(\log(10^5))}\left(1+\frac{1.4573}{\log(\log(10^5))}\right)+1}\left(1.7811\log(\log(10^5))
      + \frac{3}{\log(\log(10^5))}\right)+1\right)
      \frac{2}{m^2\log(10^5)\log(\log(10^5))}\right]^{8}\frac{1}{10^5}
    \\&<1.
  \end{align*}
  as well as
  \begin{align*}
    \left(\frac{12}{2}\right)^4 < q,
  \end{align*}
  \begin{align*}
    17^4 < q
  \end{align*}
  and
  \begin{align*}
    \left(\frac{16\log(q)}{m^2\log(\log(q))}\right)^8\frac{1}{q}<1.
  \end{align*}
  So (\ref{BordignonThm3.4}) holds. 
	
  Checking (\ref{vleqm}), we have for $v_{1}(m,q)$:
  \begin{align*}
    v_1(m,q) &\leq \frac{2\left(1+\frac{2}{e\log(q)}\right)}{27} \leq
    \frac{2\left(1+\frac{2}{e\log(10^5)}\right)}{27}\leq 0.0789,
  \end{align*}
  and for $v_2(m,q,g)$:
  \begin{align*}
    v_2(m,q,g) &\leq  0.0789^4 + \frac{2 \log(\log(q))^2}{\log(q)^2} \log(1.85\cdot 0.0789^2 \frac{q^{\frac38}\log(q)}{2\log(\log(q))} ) \\
    &\leq  0.0789^4 + \frac{2 \log(\log(10^5))^2}{\log(10^5)^2} \log(1.85\cdot 0.0789^2 \frac{(\exp(\exp(8)))^{\frac38}\exp(8)}{16} ) \leq 100.776.
  \end{align*}
  Notice here, that the above is absolutely not optimal, since we are
  estimating quite brutally using both the upper and lower limit for
  $q$. This means that the $m$ value $27$ likely is not optimal. The
  function above is eventually decreasing in $q$, so for intervals of
  large $q$, an upper bound can be obtained from plugging in just the
  lower bound for $q$.
	
  And finally for $v_3(m,q,g,h)$:
  \begin{align*}
    v_3(m,q,g,h) &\leq 4\left(\frac{16}{17}-\frac{2
    }{q^{\frac18}\log(q)\log(\log(q))}\right)^{-1}\left(\frac{17 \cdot
      100.776}{4\cdot
      2^3}\right)^{\frac14}\left(e^C+\frac{2.51}{\log(\log(q))^2}\right)
    + \frac{1}{\sqrt{2}} \\ &\leq 4\left(\frac{16}{17}-\frac{2
    }{10^{\frac58}\log(10^5)\log(\log(10^5))}\right)^{-1}\left(\frac{17
      \cdot 100.776}{4\cdot
      2^3}\right)^{\frac14}\left(e^C+\frac{2.51}{\log(\log(10^5))^2}\right)
    + \frac{1}{\sqrt{2}}\\ &\leq 26.7236... < 27.
	\end{align*}
\end{proof}

We will also need to consider Dirichlet characters which are not
primitive. To extend to this case, we employ the following, as
remarked by Burgess in \cite{Burgess1963}:
\begin{rem}\label{burgessallcharacters}
  If $\chi$ mod $q$ is not primitive, it is induced by a primitive
  character, so we may write $\chi = \chi_1\chi_2$, where $\chi_1$ is
  primitive mod $k_1$ and $\chi_2$ is principal mod $k_2$
  ($k_1+k_2=q$), and then
  \begin{align*}
    \sum_{M<n\leq M+N} \chi(n) &= \sum_{\substack{M<n\leq M+N\\ \gcd(n,k_2)=1}}\chi_1(n) \\
    &=\sum_{M<n\leq M+N} \chi_1(n)\sum_{t\mid \gcd(n,k_2)}\mu(t)
  \end{align*}
  where $\mu$ is the M\"obius function.
  Applying the relevant lemma above (depending on $q$), we get
  \begin{align}
    \left\lvert \sum_{M<n\leq M+N} \chi(n)\right\rvert &\leq
    \sum_{t\mid k_2} \left\lvert \sum_{M<ut\leq M+N}
    \chi(ut)\right\rvert =\sum_{t\mid k_2} \left\lvert
    \sum_{u=\frac{M+1}{t}}^{\frac{M+N}{t}} \chi(u) \right\rvert
    \nonumber \\ &\leq \begin{cases} 27 \cdot
      d(q)^{\frac52}\sqrt{N}q^{\frac{3}{16}}
      \sqrt{\log(q)\log(\log(q))}& \text{for }10^5 < q \leq
      \exp(\exp(8)), \\ 5 \cdot d(q)^{\frac52}\sqrt{N}q^{\frac{3}{16}}
      \sqrt{\log(q)\log(\log(q))}& \text{for }\exp(\exp(8))<q.
    \end{cases}  \label{burgessallqallchar}
  \end{align}
\end{rem}

For $q=p$ a prime, we have better explicit Burgess bounds (see also
\cite{Booker2006}, \cite{IwaniecKowalski2004} and \cite{Trevino2015})
\begin{lem}[$q=p$. Theorem 7.1 in
    \cite{McGown2012}]\label{lemmaMcGown} Suppose $\chi$ is a
  non-principal Dirichlet character modulo a prime $p\geq 2\cdot
  10^4$. For $M\in\mathbb{N}$ and $N\leq p^{\frac{5}{8}}$ we have
  \begin{equation*}
    \left\lvert \sum_{M<n\leq M+N} \chi(n) \right\rvert \leq 10.0366 \log(p)^{\frac14} p^{\frac{3}{16}} \sqrt{N}= 10.0366 \sqrt[4]{C_{\varepsilon}}\sqrt{N}p^{\frac{3}{16}+\frac{\varepsilon}{4}},
  \end{equation*}
  with $C_\varepsilon = \log(p)/p^{\varepsilon}$.
\end{lem}

Running the \texttt{Python} code in Appendix \ref{SectionPython} up to $q=20000$ yields the constant $3.41105405758246$.

\section{Proofs}\label{SectionProofs} %--------PROOFS--------
\subsection{Proof of Theorem \ref{thmallq}}\label{SectionAllq} %--------PROOF
                                                               %FOR
                                                               %ALL
                                                               %q--------

Without loss of generality, assume for any fixed $0<\eta<\frac18$, that
\[\psi(q)\geq q^{\sfrac{-5}{8}+\eta} \quad\text{ for all }q .\]
Otherwise, we replace $\psi$ with $\tilde{\psi}(q) = \max\{\psi(q),
q^{\sfrac{-5}{8}+\eta}\}$ and note that the asociated series still
converges. Note also, that $\mathcal{B}'(\tilde{\psi}, \kappa)
\subseteq \mathcal{B}'(\psi, \kappa)$, so a lower bound on the measure
of the former is automatically a lower bound on the measure of the set
of interest.  Write
\begin{equation*}
  \mathcal{B}'(\psi, \kappa)^c = \left \{x\in[0,1] : \exists
  q\in\mathbb{N} \text{ s.t. } \lVert qx\rVert, \lVert qx^2 \rVert
  \leq \kappa\psi(q) \right \}.
\end{equation*}
Note that if $\lVert qx \rVert, \lVert q x^2 \rVert \leq
\kappa\psi(q)$, there are integers $1\leq a,b\leq q$ such that
\begin{align*}
  &\lvert qx-a\rvert \leq \kappa\psi(q) \quad \text{and} \\ &\lvert
  qx^2-b\rvert \leq \kappa\psi(q).
\end{align*}
Then
\begin{align*}
  \left\lvert \frac{a^2}{q^2}-\frac{b}{q} \right\rvert &\leq
  \left\lvert \frac{a^2}{q^2} - x^2\right\rvert + \left\lvert x^2 -
  \frac{b}{q} \right\rvert \\ &= \left\lvert x-\frac{a}{q}\right\rvert
  \left\lvert x+\frac{a}{q} \right\rvert + \left\lvert
  x^2-\frac{b}{q}\right\rvert \leq 3\kappa\psi(q).
\end{align*}
From this we get, that
\begin{align}
  \lambda\left (\mathcal{B}'(\psi, \kappa)^c\right ) &\leq
  \lambda\left( \bigcup_{q=1}^{\infty} \bigcup_{\substack{a\leq q
      \\ \lVert \sfrac{a^2}{q}\rVert \leq 3\kappa\psi(q)}}
  \left[\frac{a}{q}-\frac{\kappa \psi(q)}{q}, \frac{a}{q} +
    \frac{\kappa \psi(q)}{q}\right] \right) \nonumber\\ &\leq 2\kappa
  \sum_{q=1}^{\infty} \frac{\psi(q)}{q} \sum_{\substack{a\leq q
      \\ \lVert \sfrac{a^2}{q}\rVert \leq 3\kappa\psi(q)}}
  1 \label{lebesguemeasure}
\end{align}
Our goal is now to estimate the counting function
$\sum_{\substack{a\leq q \\ \lVert \sfrac{a^2}{q}\rVert \leq
    3\kappa\psi(q)}} 1$. This function roughly counts the number of
rational points with fixed denominator $q$ lying close to the
parabola.

Put $J = \left\lfloor \frac{1}{6\kappa\psi(q)}\right\rfloor$, whenever
$\psi(q)\neq 0$. Consider the Fej\' er kernel (see
e.g. \cite{Montgomery1991})
\begin{equation*}
  \mathcal{F}_J(x) := \sum_{j=-J}^{J} \frac{J-\lvert j\rvert}{J^2}
  \exp(2\pi i j x) = \frac{1}{J^2} \left( \frac{\sin(2\pi J
    x)}{\sin(2\pi x)}\right)^2 \geq 0 \quad (x\in\mathbb{R}).
\end{equation*}
Using the local bound
\begin{align*}
  \frac{2}{\pi} \lvert x\rvert \leq \lvert \sin(x)\rvert \leq \lvert
  x\rvert \quad \text{ for } \lvert x\rvert \leq \frac{\pi}{2},
\end{align*}
we get for $\lVert \frac{a^2}{q} \rVert \leq 3\kappa\psi(q)$, that
$\left\lvert \pi \lVert\frac{a^2}{q}\rVert \right\rvert
\leq\left\lvert J\pi \lVert \frac{a^2}{q} \rVert \right\rvert \leq
\left\lvert 3 J\pi \kappa\psi(q) \right\rvert \leq \frac{\pi}{2}$ and
thus
\begin{equation*}
  \mathcal{F}_J\left(\frac{a^2}{q}\right) \geq \frac{1}{J^2}
  \left(\frac{\frac{2}{\pi}2\pi J \left
    \lVert\dfrac{a^2}{q}\right\rVert}{2\pi \left
    \lVert\dfrac{a^2}{q}\right\rVert}\right)^2 = \frac{4}{\pi^2}.
\end{equation*}
This in turn gives us, that
\begin{align}
  \sum_{\substack{a\leq q \\ \lVert \sfrac{a^2}{q} \rVert \leq
      3\kappa\psi(q)}} 1 &\leq \frac{\pi^2}{4} \sum_{\substack{a\leq
      q\\\lVert \sfrac{a^2}{q}\rVert \leq
      3\kappa\psi(q)}}\sum_{j=-J}^{J} \frac{J-\lvert j\rvert}{J^2}
  \exp\left(2\pi i j \frac{a^2}{q}\right) \notag\\ & \leq
  \frac{\pi^2}{4} \left[ \sum_{j=1}^{J} \frac{J-j}{J^2}\sum_{a=1}^{q}
    \left( \exp\left(2\pi i j\frac{a^2}{q}\right) + \exp\left(-2\pi i
    j\frac{a^2}{q}\right)\right) + \frac{q}{J} \right] \notag\\ &\leq
  \frac{\pi^2}{4} \left[ \sum_{d\mid q} d \sum_{\substack{j_1=1
        \\ \gcd(j_1, q_1)=1}}^{\sfrac{J}{d}} \frac{J-j_1
      d}{J^2}\sum_{a=1}^{q_1} \left( \exp\left(2\pi i
    j_1\frac{a^2}{q_1}\right) + \exp\left(-2\pi i
    j_1\frac{a^2}{q_1}\right)\right) + \frac{q}{J}
    \right]. \label{countingfunctionfirst}
\end{align}

To estimate this sum, we use Abel summation. Recall the following lemma.
\begin{lem}[Abel's identity, Thm. 4.2 in \cite{Apostol1979}]\label{Abel}
  Let $(a_n)_{n=0}^{\infty}$ be a sequence and for $N\in (0,\infty)$
  $\psi$ be a continuously differentiable function on $[0,N]$. Then
  \[\sum_{n=0}^{N}a_n \phi(n) = \sum_{n=0}^{N} a_n \phi(N) - \int_{0}^{N} \sum_{n=0}^{u} a_n \phi ' (u)  \mathrm{d} u.\]
\end{lem}
Taking in our case $N=\frac{J}{d}$, \(a_{j_1} = \mathbbm{1}_{\gcd(j_1,
  q_1)=1}\sum_{a=1}^{q_1}\left(\exp\left(2\pi i j_1
\frac{a^2}{q_1}\right) - \exp\left(-2\pi i j_1
\frac{a^2}{q_1}\right)\right)\) and $\phi(j_1) = \frac{J-j_1 d}{J^2}$
and using $\phi(\frac{J}{d})=\frac{J-\frac{J}{d} d}{J^2} = 0$ and
$\phi ' (u) = \frac{-d}{J^2}$ yields

\begin{equation}\label{countingfunctionexpression}
  (\ref{countingfunctionfirst}) = \frac{\pi^2}{4} \left[ \sum_{d\mid
      q} \frac{d^2}{J^2} \int_{1}^{\sfrac{J}{d}} \sum_{\substack{j_1=1
        \\ \gcd(j_1, q_1)=1}}^{u} \sum_{a=1}^{q_1} \left(
    \exp\left(2\pi i j_1\frac{a^2}{q_1}\right) + \exp\left(-2\pi i
    j_1\frac{a^2}{q_1}\right)\right) \mathrm{d} u+ \frac{q}{J} \right]
  .
\end{equation}
As in \cite{Huang2020}, to evaluate the quadratic Gauss sums, we state
the following lemma, originally from \cite{IwaniecKowalski2004} \S
3.5.
\begin{lem}\label{huanglemma1}
  Suppose $\gcd(j_1, q_1)=1$, then
  \begin{equation*}
    \sum_{a=1}^{q_1} \exp\left(2\pi i j_1\frac{a^2}{q_1}\right)
    = \begin{cases} 0 & \text{when } q_1\equiv 2 \mod
      4,\\ \left(\frac{j_1}{q_1}\right) \sqrt{q_1} & \text{when }
      q_1\equiv 1 \mod 4,\\ i\left(\frac{j_1}{q_1}\right) \sqrt{q_1} &
      \text{when } q_1\equiv 3 \mod
      4,\\ (1+i)\left(\frac{q_1}{j_1}\right)\sqrt{q_1} & \text{when }
      q_1\equiv 0 \mod 4 \text{ and } j_1\equiv 1 \mod
      4,\\ (1+i)\frac{1}{i}\left(\frac{q_1}{j_1}\right)\sqrt{q_1} &
      \text{when } q_1\equiv 0 \mod 4 \text{ and } j_1\equiv 3 \mod
      4.\\
    \end{cases}
  \end{equation*}
  Here $\left(\frac{\cdot}{\cdot}\right)$ is the Jacobi symbol.
\end{lem}

This immediately gives us,
\begin{align}
  \sum_{a=1}^{q_1}& \exp\left(2\pi i j_1\frac{a^2}{q_1}\right) +
  \overline{ \sum_{a=1}^{q_1} \exp\left(2\pi i
    j_1\frac{a^2}{q_1}\right) } \notag \\ &= \begin{cases} 0 &
    \text{when } q_1\equiv 2 \mod 4,\\ \left(\frac{j_1}{q_1}\right)
    \sqrt{q_1} + \overline{\left(\frac{j_1}{q_1}\right) \sqrt{q_1}} &
    \text{when } q_1\equiv 1 \mod 4,\\ i\left(\frac{j_1}{q_1}\right)
    \sqrt{q_1}+\overline{i\left(\frac{j_1}{q_1}\right) \sqrt{q_1}} &
    \text{when } q_1\equiv 3 \mod
    4,\\ (1+i)\left(\frac{q_1}{j_1}\right)\sqrt{q_1}
    +\overline{(1+i)\left(\frac{q_1}{j_1}\right)\sqrt{q_1} }&
    \text{when } q_1\equiv 0 \mod 4 \text{ and } j_1\equiv 1 \mod
    4,\\ (1+i)\frac{1}{i}\left(\frac{q_1}{j_1}\right)\sqrt{q_1}
    +\overline{(1+i)\frac{1}{i}\left(\frac{q_1}{j_1}\right)\sqrt{q_1}}
    & \text{when } q_1\equiv 0 \mod 4 \text{ and } j_1\equiv 3 \mod
    4,\\
  \end{cases} \notag
  \\ &= \begin{cases} 0 & \text{when } q_1\equiv 2,3 \mod
    4,\\ 2\left(\frac{j_1}{q_1}\right) \sqrt{q_1} & \text{when }
    q_1\equiv 1 \mod 4,\\ 2\left(\frac{q_1}{j_1}\right)\sqrt{q_1} &
    \text{when } q_1\equiv 0 \mod 4. \label{mylemma1} \\
  \end{cases}
\end{align}
The Legendre symbols, $\Big(\frac{\cdot}{q_1}\Big) \text{ and
}\Big(\frac{q_1}{\cdot}\Big)$, are Dirichlet characters of modulus at
most $q_1$ resp. $4q_1$.

\begin{lem}[Divisor function bound (from Terence Tao's blog \cite{Tao2008})]\label{divisorboundtao}
  For any $\beta > 0$, there is an explicit constant
  $K_{\beta}>0$ such that
  \begin{equation*}
    d(q) \leq K_{\beta} q^{\beta},
  \end{equation*}
  Specifically we have $K_{\frac{5}{11}} \leq
  \left(\frac{11}{5\log(2)}\right)^{4}$ and $K_{\frac{1}{57}} \leq
  \left(\frac{57}{\log(2)}\right)^{176846309399143769411680}$.
\end{lem}
We will repeat Tao's proof below.
\begin{proof}
  If $q$ has the prime factorisation $q=p_1^{a_1}\cdot\dots\cdot
  p_k^{a_k}$, then
  \begin{equation*}
    \frac{d(q)}{q^{\beta}} =
    \prod_{j=1}^{k}\frac{d(p_j^{a_j})}{p_j^{a_j \beta}} =
    \prod_{j=1}^{k}\frac{a_j+1}{p_j^{a_j\beta}}
  \end{equation*}
  The numerator is linear, while the denominator is exponential in
  $a_j$, so one would expect the denominator to dominate. However,
  $\beta$ might be small enough that this only happens once $p_j$
  reaches a certain size. We therefore consider small and large primes
  separately.
	
  If $p_j$ is large, say $p_j\geq e^{\frac{1}{\beta}}$, then
  \begin{equation*}
    p_j^{a_j\beta}\geq e^{a_j}\geq a_j+1
  \end{equation*}
  so indeed the denominator dominates the numerator, and the term
  contributes $\leq 1$ to the product, and may be ignored. For a fixed
  $\beta>0$ we need only consider primes $p_j < e^{\frac{1}{\beta}}$.
	
  By Taylor expansion
  \begin{equation*}
    p_j^{a_j\beta} = e^{a_j \beta \log(p_j)} \geq 1+\beta a_j \log(p_j)
  \end{equation*}
  so
  \begin{equation}\label{betalog}
    \frac{a_j + 1}{p_j^{a^j\beta}} \leq
    \frac{a_j+1}{1+a_j\beta\log(p_j)} \leq \frac{1}{\beta \log(2)}.
  \end{equation}
  
  So in the end we get
  \begin{equation*}
    d(q)\leq q^{\beta} \left(\frac{1}{\beta \log(2)}\right)^{\pi(e^{\frac{1}{\beta}})},
  \end{equation*}
  where $\pi(e^{\frac{1}{\beta}})$ denotes the number of primes
  smaller than $e^{\frac{1}{\beta}}$. We will need it for $\beta =
  \frac{5}{11}$ and $\beta = \frac{1}{57}$, where
  $\pi(e^{\frac{11}{5}}) =\pi(9)= 4$ and $\pi(e^{57}) \leq
  \pi(10^{25}) = 176846309399143769411680$ (\cite{Büthe2016}). For
  $\beta$ very small use the crude bound $\pi(e^{\frac{1}{\beta}}) <
  e^{\frac{1}{\beta}} $.
  
  Note that in (\ref{betalog}) we technically use that $\beta \leq
  \frac{1}{\log(2)}$, but this is okay since we are interested in the
  result for small $\beta$. For large $\beta$ use e.g. the trivial
  bound $d(q)\leq q$, which immediately yields the result for $\beta
  \ge 1$.
\end{proof}

\begin{lem}[Log bound]\label{logbound}
  For any $\varepsilon > 0 $
  \begin{equation*}
    \log(q)\leq C_{\varepsilon} q^{\varepsilon},
  \end{equation*}
  where
  \begin{equation*}
    C_{\varepsilon} = \frac{1}{\varepsilon e}.
  \end{equation*}
\end{lem}

\begin{proof}
  This is a simple matter of maximising
  $\frac{\log(q)}{q^{\varepsilon}}$:
  \begin{equation*}
    \dv{q}\frac{\log(q)}{q^{\varepsilon}} =
    \frac{1-\varepsilon\log(q)}{q^{1+\varepsilon}} = 0
  \end{equation*}
  for $q=e^{\frac{1}{\varepsilon}}$.  Plugging in this value in gives
  us the maximum of $\frac{\log(q)}{q^{\varepsilon}}$:
  \begin{equation*}
    \frac{\log(q)}{q^{\varepsilon}}\leq\frac{1}{\varepsilon e} \quad \forall q.
  \end{equation*}
\end{proof}

Combining our explicit Burgess bound (see Section
\ref{SectionBurgess}) with (\ref{mylemma1}), we prove a version of
Huang's Lemma 3 from \cite{Huang2020}:
\begin{lem}\label{smalllargelemma}
  For any $\varepsilon_1, \varepsilon_2 > 0$
  \begin{align}
    \sum_{\substack{j_1=1 \\ \gcd(j_1, q_1)=1}}^{N} & \sum_{a=1}^{q_1}
    \exp\left(2 \pi i j_1 \frac{a^2}{q_1}\right) + \exp\left( -2 \pi i
    j_1 \frac{a^2}{q_1}\right) \notag \\ & \leq \begin{cases} 2
      N\sqrt{q_1}& \text{if } q_1 \text{ is a square, }\\ 2\cdot
      4.19253564643679\cdot\sqrt{N} q_1^{\frac{11}{16}} & \text{if }
      q_1 \text{ is not a square and } 1\leq q_1\leq 10^5, \\ 2\cdot
      27\cdot 4^{\frac{3}{16}+\frac{\varepsilon_1+\varepsilon_1^2}{2}}
      d(4q_1)^{\frac52} C_{\varepsilon_1}^{1+\frac{\varepsilon_1}{2}}
      \sqrt{N} q_1^{\frac{11}{16}
        +\frac{\varepsilon_1+\varepsilon_1^2}{2}} & \text{if } q_1
      \text{ is not a square and } 10^5 < q_1 \leq \exp(\exp(8)),
      \\ 2\cdot 5\cdot
      4^{\frac{3}{16}+\frac{\varepsilon_2+\varepsilon_2^2}{2}}
      d(4q_1)^{\frac52} C_{\varepsilon_2}^{1+\frac{\varepsilon_2}{2}}
      \sqrt{N} q_1^{\frac{11}{16}
        +\frac{\varepsilon_2+\varepsilon_2^2}{2}} & \text{if } q_1
      \text{ is not a square and } \exp(\exp(8))< q_1.
    \end{cases} \notag 
  \end{align}
\end{lem}

\begin{proof}
  If $q_1$ is a square, then applying Lemma \ref{huanglemma1} gives us
  \begin{align*}
    \sum_{\substack{j_1=1 \\ \gcd(j_1, q_1)=1}}^{N} \sum_{a=1}^{q_1}
    \exp\left(2 \pi i j_1 \frac{a^2}{q_1}\right) + \exp\left( -2 \pi i
    j_1 \frac{a^2}{q_1}\right) &\leq 2\sqrt{q_1}\sum_{\substack{j_1=1
        \\ \gcd(j_1, q_1)=1}}^{N} \chi(j_1) \\ &\leq 2\sqrt{q_1} N,
  \end{align*}
  where $\chi(j_1) = \left(\frac{j_1}{q_1}\right)$ if $q_1 \equiv 1
  \mod 4$ and $\chi(j_1) = \left(\frac{q_1}{j_1}\right)$ if $q_1
  \equiv 0 \mod 4$.
	
  For $q_1$ not a square and $q_1 > 10^5$, using Lemma
  \ref{huanglemma1}, we have
  \begin{align*}
    \sum_{\substack{j_1=1 \\ \gcd(j_1, q_1)=1}}^{N} \sum_{a=1}^{q_1}
    \exp\left(2 \pi i j_1 \frac{a^2}{q_1}\right) + \exp\left( -2 \pi i
    j_1 \frac{a^2}{q_1}\right) \leq 2 \sqrt{q_1} \sum_{\substack{j_1=1
        \\ \gcd(j_1, q_1)=1}}^{N} \chi(j_1),
  \end{align*}
  where $\chi(j_1) = \left(\frac{j_1}{q_1}\right)$ if $q_1 \equiv 1
  \mod 4$, $\chi(j_1) = \left(\frac{q_1}{j_1}\right)$ if $q_1 \equiv 0
  \mod 4$ and is zero otherwise. Either way $\chi$ is a Dirichlet
  character of modulus $\leq 4q_1$. Applying Remark
  \ref{burgessallcharacters} as well as Lemma \ref{logbound}, we get
  for $q_1$ not a square and $10^5<q_1\leq \exp(\exp(8))$:
  \begin{align*}
    \leq 2\cdot
    4^{\frac{3}{16}+\frac{\varepsilon_!+\varepsilon_1^2}{2}} 27
    C_{\varepsilon_1}^{1+\frac{\varepsilon_1}{2}} \sqrt{N}
    d(4q_1)^{\frac52} q_1^{\frac{11}{16} +
      \frac{\varepsilon_1+\varepsilon_1^2}{2} }
	\end{align*}
  and for $q_1$ not a square and $\exp(\exp(8)) < q_1$:
  \begin{align*}
    \leq 2\cdot
    4^{\frac{3}{16}+\frac{\varepsilon_2+\varepsilon_2^2}{2}} 5
    C_{\varepsilon_2}^{1+\frac{\varepsilon_2}{2}}
    \sqrt{N}d(4q_1)^{\frac52} q_1^{\frac{11}{16} +
      \frac{\varepsilon_2+\varepsilon_2^2}{2} }.
  \end{align*}
	
  For $q_1$ not a square and $1\leq q_1 \leq 10^5$ apply Lemma
  \ref{huanglemma1} and then Lemma \ref{burgesssmallq}.
\end{proof}

We now split the sum over $d\mid q$ in
(\ref{countingfunctionexpression}) into the four cases,
\begin{enumerate}
  \item $q=q_1 d$ and $q_1$ is a square, denoted $q_1 = \square$
    below,
  \item $q=q_1 d$, $q_1$ is not a square,denoted $q_1 \neq \square$, and $1\leq
    q_1\leq 10^5$,
  \item $q=q_1 d$, $q_1$ is not a square, and $10^5 <
    q_1\leq \exp(\exp(8))$ and
  \item $q=q_1 d$, $q_1$ is not a square, and $\exp(\exp(8)) < q_1$.
\end{enumerate}
We then apply Lemma \ref{smalllargelemma}, and in the case $q=q_1 d$,
$q_1$ is not a square, and $\exp(\exp(8)) < q_1$ Lemma
\ref{divisorboundtao} for some $\beta_2 >0$ to obtain
\begin{align}
  (\ref{countingfunctionexpression}) &\leq \frac{\pi^2}{4}\left[
    2\sum_{q=q_1 d, q_1=\square} \frac{d^2}{J^2} \sqrt{q_1}
    \int_{1}^{\frac{J}{d}} u ~\mathrm{d}u + 2\cdot 4.19253564643679
    \sum_{q=q_1d, q_1\neq \square, 1\leq q_1\leq 10^5} \frac{d^2}{J^2}
    q_1^{\frac{11}{16}}\int_{1}^{\frac{J}{d}} \sqrt{u} ~\mathrm{d}u
    \right.\nonumber\\ &\qquad \left. {} + 2\cdot
    4^{\frac{3}{16}+\frac{\varepsilon_1+\varepsilon_1^2}{2}} 27
    C_{\varepsilon_1}^{1+\frac{\varepsilon_1}{2}} \sum_{q=q_1d,
      q_1\neq\square, 10^5 < q_1\leq \exp(\exp(8))} \frac{d^2}{J^2}
    q_1^{\frac{11}{16}+\frac{\varepsilon_1+\varepsilon_1^2}{2}}
    d(4q_1)^{\frac52}\int_{1}^{\frac{J}{d}} \sqrt{u} ~ \mathrm{d}u
    \right.\nonumber\\ &\qquad \left. {} + 2\cdot
    4^{\frac{3}{16}+\frac{5\beta_2}{2}+\frac{\varepsilon_1+\varepsilon_1^2}{2}}
    5 K_{\beta_2}^{\frac{5}{2}}
    C_{\varepsilon_2}^{1+\frac{\varepsilon_2}{2}} \sum_{q=q_1d,
      q_1\neq\square, \exp(\exp(8)) < q_1} \frac{d^2}{J^2}
    q_1^{\frac{11}{16}+\frac{5\beta}{2}+\frac{\varepsilon_2+\varepsilon_2^2}{2}}
    \int_{1}^{\frac{J}{d}} \sqrt{u} ~ \mathrm{d}u +\frac{q}{J} \right]
  \notag\\ &\leq \frac{\pi^2}{4}\left[ \sum_{q=q_1 d, q_1=\square}
    \sqrt{q_1} \frac{d^2}{J^2} \frac{J^2}{d^2} + 2\cdot
    4.19253564643679 \frac{2}{3}\sum_{q=q_1d, q_1\neq \square, 1\leq
      q_1\leq 10^5} q_1^{\frac{11}{16}} \frac{d^2}{J^2}
    \left(\frac{J}{d}\right)^{\frac{3}{2}} \right.\nonumber \\ &\qquad
    \left. {} + 2\cdot 4^{\frac{3}{16}
      +\frac{\varepsilon_1+\varepsilon_1^2}{2}} 27
    C_{\varepsilon_1}^{1+\frac{\varepsilon_1}{2}}
    \frac{2}{3}\sum_{q=q_1d, q_1\neq\square, 10^5 < q_1 \leq
      \exp(\exp(8))}
    q_1^{\frac{11}{16}+\frac{\varepsilon_1+\varepsilon_1^2}{2}}
    d(4q_1)^{\frac52} \frac{d^2}{J^2}
    \left(\frac{J}{d}\right)^{\frac{3}{2}} \right.\nonumber \\ &\qquad
    \left. {} + 2\cdot 4^{\frac{3}{16}
      +\frac{5\beta_2}{2}+\frac{\varepsilon_2+\varepsilon_2^2}{2}} 5
    K_{\beta_2}^{\frac52}
    C_{\varepsilon_2}^{1+\frac{\varepsilon_2}{2}}
    \frac{2}{3}\sum_{q=q_1d, q_1\neq\square, \exp(\exp(8)) < q_1}
    q_1^{\frac{11}{16}+\frac{5\beta_2}{2}+\frac{\varepsilon_2+\varepsilon_2^2}{2}}
    \frac{d^2}{J^2} \left(\frac{J}{d}\right)^{\frac{3}{2}} +
    \frac{q}{J} \right]. \label{countingfunctionexpression2}
\end{align}

We have $d(4q_1)\leq 3d(q_1)$. In the case $q=q_1 d$, $q_1$ is not a
square, and $10^5 < q_1\leq \exp(\exp(8))$, applying Lemma
\ref{divisorboundlargeq} to $d(q_1)$ yields (for $10^5\leq q_1\leq
\exp(\exp(8))$):
\begin{equation*}
  \log(d(q_1)) \leq \frac{\log(2)\log(q)}{\log(\log(q))-1.39177} \leq
  \frac{\log(2)e^8}{8-1.39177},
\end{equation*}
so
\begin{equation*}
  d(4q_1)^{\frac52} \leq 3^{\frac52} e^{\frac{5\log(2)e^8}{2\left(8-1.39177\right)}}.
\end{equation*}
Employing the bounds 
\begin{equation*}
  \sum_{\substack{q=q_1 d \\ \substack{q_1 \neq \square \\ q_1\leq
        10^5}}} 1\leq 99684, \quad
  \sum_{\substack{q=q_1d\\ \substack{q_1\neq\square \\ 10^5 < q_1 \leq
        \exp(\exp(8))}}} 1 \leq (\exp(\exp(8))-10^5), \text{ and
  }\sum_{\substack{q=q_1 d \\ \substack{q_1 \neq \square
        \\ \exp(\exp(8)) < q_1 }}} 1 \leq d(q) \leq K_{\beta_2}
  q^{\beta_2}.
\end{equation*}
gives us the following upper bound for the counting function:
\begin{align}
  (\ref{countingfunctionexpression2}) &\leq \frac{\pi^2}{4}\left[
    \sum_{q=q_1 d, q_1=\square}\sqrt{q_1} + 2\cdot 4.19253564643679
    \frac{2}{3} J^{\frac{-1}{2}} q^{\frac{11}{16}} 99684
    \right.\nonumber \\ &\qquad \left. {} + 2\cdot 27\cdot
    4^{\frac{3}{16}+\frac{\varepsilon_1+\varepsilon_1^2}{2}}
    C_{\varepsilon_1}^{1+\frac{\varepsilon_1}{2}}
    \frac{2}{3}J^{\frac{-1}{2}}q^{\frac{11}{16}+\frac{\varepsilon_1+
        \varepsilon_1^2}{2}} 3^{\frac52}
    e^{\frac{5\log(2)e^8}{2\left(8-1.39177\right)}}
    \left(\exp(\exp(8))-10^5\right) \right.\nonumber \\ &\qquad
    \left. {} + 2\cdot 5\cdot
    4^{\frac{3}{16}+\frac{5\beta_2}{2}+\frac{\varepsilon_2+\varepsilon_2^2}{2}}
    K_{\beta_2}^{\frac{7}{2}}C_{\varepsilon_2}^{1+\frac{\varepsilon_2}{2}}
    \frac{2}{3}J^{\frac{-1}{2}}q^{\frac{11}{16}+\frac{7\beta_2}{2}+\frac{\varepsilon_2+
        \varepsilon_2^2}{2}} + \frac{q}{J} \right] \notag
\end{align}

For $r$ the biggest square dividing $q$, applying Lemma
\ref{divisorboundtao} for some $\beta_3>0$ yields
\begin{align*}
  \sum_{\substack{q=q_1 d \\ q_1 = \square}} \sqrt{q_1} =
  \sum_{q=r_1^2d}r_1 \leq r\sum_{r_1\mid r} 1 = r d(r) \leq
  r^{1+\beta_3} K_{\beta_3}.
\end{align*}
Furthermore, since $J=\left\lfloor \frac{1}{6\kappa\psi(q)}
\right\rfloor \left(\neq 0\right)$, we have $J^{-1} \leq
12\kappa\psi(q).$ So
\begin{align*}
  \lambda(\mathcal{B}'(\psi, \kappa)^c) &\leq \kappa
  K_{\beta_3}\frac{\pi^2}{2}\sum_{q=1}^{\infty}
  \frac{\psi(q)}{q}r^{1+\beta_3} \nonumber \\ & {} + \kappa^{\frac32}
  \left( \frac{2 \pi^2 4.19253564643679\sqrt{12}99684}{3}
  \sum_{q=1}^{\infty} \psi(q)^{\frac32} q^{\frac{-5}{16}}
  \right.\nonumber \\ &\qquad \left. {} +\frac{2 \pi^2
    4^{\frac{3}{16}+\frac{\varepsilon_1 +
        \varepsilon_1^2}{2}}27\sqrt{12}C_{\varepsilon_1}^{1+\frac{\varepsilon_1}{2}}
    3^{\frac52} e^{\frac{5\log(2)e^8}{2\left(8-1.39177\right)}}
    \left(\exp(\exp(8))-10^5\right) }{3} \sum_{q=1}^{\infty}
  \psi(q)^{\frac32} q^{\frac{-5}{16}+
    \frac{\varepsilon_1+\varepsilon_1^2}{2}} \right.\nonumber
  \\ &\qquad \left. {} +\frac{2\pi^2
    4^{\frac{3}{16}+\frac{\varepsilon_2+\varepsilon_2^2}{2}+\frac{5\beta_2}{2}}5\sqrt{12}C_{\varepsilon_2}^{1+\frac{\varepsilon_2}{2}}K_{\beta_2}^{\frac{7}{2}}}{3}
  \sum_{q=1}^{\infty} \psi(q)^{\frac32}
  q^{\frac{-5}{16}+\frac{\varepsilon_2+\varepsilon_2^2}{2}+\frac{7\beta_2}{2}}
  \right) \nonumber \\ & {} + \kappa^2 6\pi^2 \sum_{q=1}^{\infty}
  \psi(q)^2.
\end{align*}

Choose $\beta_2 = \frac{1}{57}$, $\beta_3 = \frac{5}{11}$,
$\varepsilon_1 = 0.11237$ and $\varepsilon_2 = 0.00218819$.  For the
first series, we have
\begin{align*}
  \sum_{q=1}^{\infty} \frac{\psi(q)}{q}r^{1+\beta_3} &\leq
  \sqrt{\sum_{q=1}^{\infty} \psi(q)^2} \sqrt{\sum_{q=1}^{\infty}
    r^{2\left(1+\beta_3\right)}q^{-2}} \\ &\leq \sqrt{S_{\psi}}
  \sqrt{\sum_{r=1}^{\infty} \sum_{t=1}^{\infty} \lvert \mu(t)\rvert
    r^{2+\frac{10}{11}} (r^2t)^{-2}} \\ &\leq \sqrt{S_{\psi}}
  \sqrt{\sum_{r=1}^{\infty} r^{\frac{-12}{11}} \sum_{t=1}^{\infty}
    t^{-2} } \\ &\leq
  \sqrt{S_{\psi}}\sqrt{\zeta\left(\frac{12}{11}\right)}\sqrt{\frac{\pi^2}{6}},
\end{align*}
where $\mu$ is the M{\"o}bius function. Note that we estimate
$\vert\mu(t)\vert$ by $1$ here. A slightly sharper estimate would
follow if one only considered square free $t$, since the function in
$0$ elsewhere, but the improvement is so marginal that we decided not
to pursue this further. For the second, third and fourth, we have
\begin{align*}
  \sum_{q=1}^{\infty} \psi(q)^{\frac32}q^{\frac{-5}{16}},
  \sum_{q=1}^{\infty}
  \psi(q)^{\frac32}q^{\frac{-5}{16}+\frac{\varepsilon_1+\varepsilon_1^2}{2}},
  \sum_{q=1}^{\infty}
  \psi(q)^{\frac32}q^{\frac{-5}{16}+\frac{\varepsilon_2+\varepsilon_2^2}{2}+\frac{7\beta_2}{2}}
  \leq S_{\psi}
\end{align*}
since $\psi(q)\geq q^{\frac{-5}{8}+\eta}$.
So $\lambda\left(\mathcal{B}'(\psi, \kappa)^c\right)$ is
\begin{align*}
  &\leq \kappa \frac{\pi^2
    \sqrt{\zeta\left(\frac{12}{11}\right)}\sqrt{\frac{\pi^2}{6}}
    \left(\frac{11}{5\log(2)}\right)^{4} }{2} \sqrt{S_{\psi}} \\ & +
  \kappa^{\frac32} \left( \frac{2\pi^2\sqrt{12} 99684 \cdot
    4.192423564643679}{3} \right.\nonumber \\ &\qquad \left. {} +
  \frac{2\pi^2 27\sqrt{12}\cdot 4^{\frac{3}{15} +
      \frac{0.11237+0.11237^2}{2}} 3^{\frac52}
    e^{\frac{5\log(2)e^8}{2(8-1.39177)}}\left(e^{e^8}-10^5\right)
    \left(\frac{1}{e\cdot 0.11237}\right)^{1+\frac{0.11237}{2}}}{3}
  \right.\nonumber \\ &\qquad \left. {} + \frac{2\pi^2 5
    \sqrt{12}\cdot 4^{\frac{3}{16} + \frac{0.00218819+0.00218819^2}{2}
      + \frac{5}{2\cdot 57}}
    \left(\frac{57}{\log(2)}\right)^{\frac{7}{2}176846309399143769411680}
    \left(\frac{1}{e\cdot
      0.00218819}\right)^{1+\frac{0.00218819}{2}}}{3 } \right)S_{\psi}
  \\ & + \kappa^2 6\pi^2 S_{\psi}
\end{align*}
which is $<\delta$ for 
\begin{align*}
  \kappa(\delta) < \min &\left\{ \frac{\delta}{\sqrt{S_{\psi}}}
  \frac{2 \sqrt{6} (5\log(2))^4}{3\pi^3
    \sqrt{\zeta\left(\frac{12}{11}\right)} 11^4},
  \sqrt{\frac{\delta}{S_{\psi}}} \frac{1}{\sqrt{18}\pi},
  \right.\nonumber\\ &\left. {}
  \left(\frac{\delta}{S_{\psi}}\right)^{\frac23} \left(\frac{1}{2\pi^2
    \sqrt{12}} \right)^{\frac23}\Bigg( 99684\cdot 4.19253564643679
  \right.\nonumber\\ & \left. {} +27\cdot 4^{\frac{3}{16} +
    \frac{0.11237+0.11237^2}{2}} 3^{\frac52} \left(e^{e^8}-10^5\right)
  e^{\frac{5\log(2) e^8}{2(8-1.39177)}}\left(\frac{1}{e\cdot
    0.11237}\right)^{1+\frac{0.11237}{2}} \right.\nonumber\\ &
  \left. {} +5\cdot 4^{\frac{3}{16}+\frac{0.00218819+0.00218819^2}{2}
    + \frac{5}{2\cdot 57}}
  \left(\frac{57}{\log(2)}\right)^{\frac{7}{2}176846309399143769411680}
  \left(\frac{1}{e\cdot
    0.00218819}\right)^{1+\frac{0.00218819}{2}}\Bigg)^{\frac{-2}{3}}
  \right\} \\ &\simeq \min\left\{ \frac{\delta}{\sqrt{S_{\psi}}}
  0.00015248265228 , \sqrt{\frac{\delta}{S_{\psi}}} 0.07502635967975,
  \left(\frac{\delta}{S_{\psi}}\right)^{\frac23}
  10^{-10^{23.89775276641198}}\right\}.
\end{align*}
\hfill $\square$

\begin{rem*}
  The constant $\sim 10^{-10^{23.89775276641198}}$ is very small. The
  main problem here is our divisor function bound from Lemma
  \ref{divisorboundtao}. There are two ways to address the
  problem. Large values of $d(q)$ occur at smooth numbers,
  i.e. numbers with only small prime divisrs, so one could consider
  e.g., for a fixed $n$, only $q$ of the form $q=p_1 p_2 \dots p_n$,
  so $d(q)\leq 2^n$. Alternatively by looking at only very large
  values of $q$, we also get a more reasonable constant in our divisor
  function bound. The former is explored in Theorem \ref{thmp1...pn}
  and Theorem \ref{thmprime} (for $n=1$), while latter is Theorem
  \ref{thmlargeq}.
\end{rem*}

\subsection{Proof of Theorem \ref{thmlargeq}}\label{SectionLargeq} %--------PROOF
%FOR
%LARGE
%q--------
If we consider only very large $q$, we can improve Lemmas
\ref{divisorboundtao} and \ref{logbound} as follows:
\begin{lem}[From Robin's thesis \cite{Robin1983}, also mentioned in the survey paper \cite{NicolasSurvey}]\label{divisorboundlargeq}
  For $q\geq 56$
  \begin{equation*}
    \log(d(q)) \leq \frac{\log(q)\log(2)}{\log(\log(q))-1.39177}
  \end{equation*}
  In particular, for $q\geq \exp(\exp(41))$, we get $d(q)\leq q^{\frac{1}{57}}$.
\end{lem}

\begin{lem}
  For $q\geq \exp(\exp(41))$
  \begin{equation*}
    \log(q)\leq q^{\frac{41}{e^{41}}}.
  \end{equation*}
\end{lem} 

\begin{proof}
  We have $\log(q) \leq q^{\frac{41}{e^{41}}}$ if and only if
  $\frac{41}{e^{41}} \geq \frac{\log(\log(q))}{\log(q)}$.
\end{proof}
Aside from having these improved lemmas, the proof is exactly the same
as before, only with $K_{\beta}$ and $C_{\varepsilon}$ both replaced
by $1$ and with $\beta$ and $\varepsilon$ replaced by $\frac{1}{57}$
respectively $\frac{41}{e^{41}}$. Picking up from
(\ref{countingfunctionexpression}) we have
\begin{align*} 
  \sum_{\substack{a\leq q \\ \lVert \sfrac{a^2}{q} \rVert \leq
      3\kappa\psi(q)}} 1 &\leq \frac{\pi^2}{4}\left[ \sum_{d\mid q}
    \frac{d^2}{J^2} \int_{1}^{\frac{J}{d}} \sum_{\substack{j_1=1
        \\ \gcd(j_1, q_1)=1}}^{u} \sum_{a=1}^{q_1} e^{2\pi i
      j_1\frac{a^2}{q_1}} + e^{-2\pi i j_1\frac{a^2}{q_1}} \mathrm{d}u
    + \frac{q}{J}\right] \nonumber \\ &\leq \frac{\pi^2}{4} \left[
    \sum_{\substack{q=q_1d \\ q_1 = \square}} \frac{d^2}{J^2}
    \int_{1}^{\frac{J}{d}} 2\sqrt{q_1}u~\mathrm{d}u \right. \\ &\qquad
    \left. {} + \sum_{\substack{q=q_1 d\\q_1\neq \square\\1\leq
        q_1\leq 10^5}} \frac{d^2}{J^2} \int_{1}^{\frac{J}{d}} 2\cdot
    4.19253564643679 q_1^{\frac{11}{16}} \sqrt{u}
    ~\mathrm{d}u\right. \\ &\qquad \left. {} + \sum_{\substack{q=q_1
        d\\q_1\neq \square\\10^5 < q_1\leq \exp(\exp(8))}}
    \frac{d^2}{J^2} \int_{1}^{\frac{J}{d}} 2\cdot 27\cdot
    4^{\frac{3}{16}+\frac{\varepsilon+\varepsilon^2}{2}}
    d(4q_1)^{\frac{5}{2}}
    q_1^{\frac{11}{16}+\frac{\varepsilon+\varepsilon^2}{2}}
    \sqrt{u}~\mathrm{d}u + \frac{q}{J}\right]
\end{align*}
This time we use the estimates
\begin{align*}
  \sum_{\substack{q=q_1 d \\ q_1 \neq \square \\ 1\leq q_1\leq 10^5}}
  1, \sum_{\substack{q=q_1 d \\ q_1 \neq \square \\ 10^5 < q_1 }} 1
  \leq d(q) < q^{\frac{1}{57}},
\end{align*}
as well as 
\begin{align*}
  q^{\frac{1}{57}} q_1^{\frac{11}{16}} \sqrt{d}\leq q^{\frac{11}{16}+\frac{1}{57}} 
\end{align*}
and
\begin{align*}
  q_1^{\frac{11}{16} + \frac{7}{2 \cdot 57}+
    \frac{\frac{41}{e^{41}}+\left(\frac{41}{e^{41}}\right)^2}{2}}
  \sqrt{d} \leq q^{\frac{11}{16} + \frac{7}{2\cdot 57}+
    \frac{\frac{41}{e^{41}}+\left(\frac{41}{e^{41}}\right)^2}{2}}.
\end{align*}
Apply also
\begin{equation*}
  \sum_{\substack{q=q_1 d \\ q_1 = \square}} \sqrt{q_1} =
  \sum_{q=r_1^2 d} r_1 \leq r\sum_{r_1\mid r} 1 \leq r d(q)\leq r
  q^{\frac{1}{57}},
\end{equation*}
where $r$ is the largest square dividing $q$.  Putting everything
together, exactly like in the proof of Theorem \ref{thmallq}, we get
\begin{align*}
  \lambda\left(\mathcal{B}'(\kappa, \psi)^c\right) &\leq \kappa
  \frac{\pi^2}{2} \sum_{q=1}^{\infty} \psi(q) q^{-1} r^{\frac{58}{57}}
  \\ &\quad {} + \kappa^{\frac{3}{2}} \frac{2\pi^2 \sqrt{12}}{3}
  \left(4.19253564643679 \sum_{q=1}^{\infty} \psi(q)^{\frac32}
  q^{\frac{-5}{16}+\frac{1}{57}} \right.\\ &\qquad \left. {} + 27\cdot
  4^{\frac{3}{16}+\frac{\frac{41}{e^{41}} +
      \left(\frac{41}{e^{41}}\right)^2}{2} + \frac{5}{2\cdot
      57}}\sum_{q=1}^{\infty}
  \psi(q)^{\frac32}q^{\frac{-5}{16}+\frac{\frac{41}{e^{41}} +
      \left(\frac{41}{e^{41}}\right)}{2} + \frac{7}{2\cdot 57}\ }
  \right)\\ & \quad {} + \kappa^2 6\pi^2 \sum_{q=1}^{\infty}
  \psi(q)^2.
\end{align*}
For the first series we have 
\begin{align*}
  \sum_{q=1}^{\infty} \psi(q) q^{-1} r^{\frac{58}{57}} &\leq
  \sqrt{\sum_{q=1}^{\infty} \psi(q)} \sqrt{\sum_{r=1}^{\infty}
    \sum_{t=1}^{\infty} \lvert \mu(t)\rvert r^{2\frac{58}{57}} (r^2
    t)^{2} } \\ &\leq \sqrt{S_{\psi}} \sqrt{ \sum_{r=1}^{\infty}
    r^{\frac{-112}{57}} \sum_{t=1}^{\infty} t^{-2} }\\ &\leq
  \sqrt{S_{\psi}}\sqrt{\zeta\left(\frac{112}{57}\right)
    \frac{\pi^2}{6}}.
\end{align*}
Since $\psi(q)\geq q^{-\frac{5}{8}+\eta}$, the other three series are
$\leq S_{\psi}$. Thus
\begin{align*}
  \lambda\left(\mathcal{B}'(\kappa, \psi)^c\right) &\leq \kappa
  \frac{\pi^2}{2}\sqrt{S_{\psi}}\sqrt{\frac{\pi^2}{6}\zeta\left(\frac{112}{57}\right)}
  + \kappa^2 6\pi^2 S_{\psi}\\ &\quad {} + \kappa^{\frac{3}{2}}
  \frac{2\pi^2 \sqrt{12}}{3} S_{\psi}\left(4.19253564643679+ 27\cdot
  4^{\frac{3}{16}+\frac{\frac{41}{e^{41}} +
      \left(\frac{41}{e^{41}}\right)^2}{2} + \frac{5}{2\cdot 57}}
  \right).
\end{align*}
This is $<\delta$ for
\begin{align*}
  \kappa(\delta) < \min&\left\{ \frac{\delta}{ \sqrt{S_{\psi}}}
  \frac{2}{3\pi^2
    \sqrt{\frac{\pi^2}{6}\zeta\left(\frac{112}{57}\right)}},
  \sqrt{\frac{\delta}{S_{\psi}}} \sqrt{\frac{1}{18\pi^2}} \right. \\ &
  \left. {} \left(\frac{\delta}{S_{\psi}}\right)^{\frac23}
  \left(2\pi^2\sqrt{12}\left(4.19253564643679+ 27\cdot
  4^{\frac{3}{16}+\frac{5}{2\cdot 57} +
    \frac{\frac{41}{e^{41}}+\left(\frac{41}{e^{41}}\right)^2}{2}}
  \right)\right)^{\frac{-2}{3}}\right\} \\ \simeq \min&\left\{
  \frac{1}{6\max_{n\in\mathbb{N}}\psi(n)},
  \frac{\delta}{\sqrt{S_{\psi}}} 0.04064400,
  \sqrt{\frac{\delta}{S_{\psi}}} 0.07502635,
  \left(\frac{\delta}{S_{\psi}}\right)^{\frac23}0.00499238\right\}.
\end{align*}
\hfill $\square$

\begin{rem*}
  The size of the constants in $\kappa(\delta)$ are much more
  managable here, but it comes at the cost of assuming $q >
  \exp(\exp(41))$, which is again very large. Once more, this is due
  to our divisor function bound!
\end{rem*}

\subsection{Proof of Theorem \ref{thmp1...pn}}%--------PROOF FOR q=p_1...p_n--------
Since $q$ has at most $n$ prime divisors, we have $d(q)\leq
2^n$. As in \S \ref{SectionAllq} equation (\ref{countingfunctionexpression2}) (but not using the divisor function bound from \S \ref{SectionAllq}) we have
\begin{align*}
  \sum_{\substack{a\leq q \\ \lVert \frac{a^2}{q}\leq 3
      \kappa\psi(q)}} 1 \leq \frac{\pi^2}{4} &
  \left[\sum_{\substack{q=q_1d\\q_1=\square}}\sqrt{q_1} + 2\cdot
    4.192535.64643679\cdot \frac23 J^{-\frac12}q^{\frac{11}{16}}99684
    \right. \\ &\left. {} +2\cdot 27\cdot
    4^{\frac{3}{16}+\frac{\varepsilon+\varepsilon^2}{2}}\left(\frac{1}{e\varepsilon}\right)^{1+\frac{\varepsilon}{2}}
    \frac23
    d(4q)^{\frac52}d(q)q^{\frac{11}{16}+\frac{\varepsilon+\varepsilon^2}{2}}
    J^{-\frac12} + \frac{q}{J} \right].
\end{align*}
Using $d(q)\leq 2^n$, $J^{-1} \leq 12\kappa\psi(q)$, and
\begin{equation*}
  \sum_{\substack{q=q_1d\\q_1=\square}}\sqrt{q_1} \leq rd(r)\leq rd(q) \leq r 2^n
\end{equation*}
gives us
\begin{align*}
  \leq &\frac{\pi^2}{4}\left[r 2^n +
    \frac{4\cdot4.19253564643679\sqrt{12}99684 }{3}\sqrt{\kappa}
    \sqrt{\psi(q)}q^{\frac{11}{16}} \right. \\ &\left. {} +
    \frac{4\cdot 27 \cdot 4^{\frac{3}{16} +
        \frac{\varepsilon+\varepsilon^2}{2}}\sqrt{12} 3^{\frac{5}{2}}
      2^{\frac{7n}{2}}
      \left(\frac{1}{e\varepsilon}\right)^{1+\frac{\varepsilon}{2}}}{3}
    \sqrt{\kappa}\sqrt{\psi(q)}q^{\frac{11}{16} +
      \frac{\varepsilon+\varepsilon^2}{2}} +12\kappa q\psi(q)\right].
\end{align*}
Plugging into (\ref{lebesguemeasure})
\begin{align*}
  \lambda\left(\mathcal{B}'(\kappa,\psi)^c\right)\leq & \kappa
  \frac{\pi^2}{2}2^n \sum_{q=1}^{\infty} \frac{\psi(q)}{q}r \\ & {} +
  \kappa^{\frac32} \frac{2 \pi^2 \sqrt{12}}{3}
  \left(4.19253564643679\cdot 99684 \sum_{q=1}^{\infty}
  \psi(q)^{\frac32} q^{\frac{-5}{16}} \right.\\ &\quad\left.  {} +
  27\cdot
  4^{\frac{3}{16}+\frac{\varepsilon+\varepsilon^2}{2}}3^{\frac52}\left(\frac{1}{e\varepsilon}\right)^{1+\frac{\varepsilon}{2}}
  2^{\frac{7n}{2}} \sum_{q=1}^{\infty} \psi(q)^{\frac32}
  q^{\frac{-5}{16}+\frac{\varepsilon+\varepsilon^2}{2}} \right)\\ & {}
  + \kappa^2 6\pi^2 \sum_{q=1}^{\infty} \psi(q)^2.
\end{align*}

Let $\varepsilon=0.11237$, then
\begin{equation*}
  \sum_{q=1}^{\infty} \frac{\psi(q)}{q}r \leq
  \sqrt{\sum_{q=1}^{\infty} \psi(q)^2}\sqrt{\sum_{r=1}^{\infty}
    \sum_{t=1}^{\infty} r^{-2} t^{-2}}
  \leq\sqrt{S_{\psi}}\frac{\pi^2}{6},
\end{equation*}
and the other sums are $\leq S_{\psi}$. All together
$\lambda\left(\mathcal{B}'(\kappa, \psi)^c\right)$ is $<\delta$ for
\begin{align*}
  \kappa(\delta) < \min&\left\{
  \frac{\delta}{\sqrt{S_{\psi}}}\frac{4}{3\pi^4}\frac{1}{2^n} ,
  \left(\frac{\delta}{S_{\psi}}\right)^{\frac{2}{3}}\left(2\pi^2\sqrt{12}\Bigg[4.19253564643679\cdot
    99684 \right.\right. \\ &\quad \left.\left. {} + 27\cdot
    4^{\frac{3}{16}+\frac{0.11237+0.11237^2}{2}}3^{\frac52}\left(\frac{1}{e\cdot
      0.11237}\right)^{1+\frac{0.11237}{2}} 2^{\frac{7n}{2}}
    \Bigg]\right)^{\frac{-2}{3}}, \sqrt{\frac{\delta}{S_{\psi}}}
  \sqrt{\frac{1}{18\pi^2}} \right\} \\ \simeq
  \min&\left\{\frac{\delta}{\sqrt{S_{\psi}}} 0.0410639\frac{1}{2^n},
  \left(\frac{\delta}{S_{\psi}}\right)^{\frac23} 0.0598025
  \left(417928.72 + 2082.9292 \cdot
  2^{\frac{7n}{2}}\right)^{\frac{-2}{3}},
  \sqrt{\frac{\delta}{S_{\psi}}} 0.0750263\right\}.
\end{align*}
\hfill $\square$

\subsection{Proof of Theorem \ref{thmprime}} %--------PROOF FOR q PRIME--------
We go a bit more into detail here, since $q=p$ being a prime affects a
lot.  As in \S \ref{SectionAllq} (\ref{countingfunctionfirst}), we have
\begin{align*}
  \sum_{\substack{a\leq p \\ \left\lVert \frac{a^2}{p} \right\rVert
      \leq 3\kappa\psi(p)}} 1 \leq \frac{\pi^2}{4}
  \left(\sum_{j=1}^{J} \frac{J-j}{J^2} \sum_{a=1}^{p} \left( e^{2\pi
    j\frac{a^2}{p}} + e^{-2\pi i\frac{a^2}{p}}\right)
  +\frac{p}{J}\right).
\end{align*}
If $J < p$, then for $1\leq j\leq J$, $\gcd(j,p)$ is guaranteed to be
$1$, but if $J\geq p$, there is a chance, that $1\leq j\leq J$ is a
multiple of $p$. For $J\geq p$ we have
\begin{align*}
  \sum_{j=1}^{J} & \frac{J-j}{J^2}\sum_{a=1}^{p} \left( e^{2\pi i j
    \frac{a^2}{p}} + e^{-2\pi i j \frac{a^2}{p}} \right) + \frac{p}{J}
  \\ &\leq 1\cdot \sum_{\substack{j=1 \\ \gcd(j,p)=1}}^{J}
  \frac{J-j}{J^2} \sum_{a=1}^{p} \left( e^{2\pi i j\frac{a^2}{p}} +
  e^{-2\pi i j\frac{a^2}{p}} \right) + p \cdot
  \sum_{\substack{j_1=1\\ \gcd(j_1, \frac{p}{p})=1}}^{\frac{J}{p}}
  \frac{J-j_1 p }{J^2} \sum_{a=1}^{\frac{p}{p}} 2 + \frac{p}{J}
  \\ &\leq \sum_{\substack{j=1 \\ \gcd(j,p)=1}} \frac{J-j}{J^2}
  \sum_{a=1}^{p} \left( e^{2\pi i j\frac{a^2}{p}} + e^{-2\pi i
    j\frac{a^2}{p}} \right) + 1.
\end{align*}
In either case, to the expression 
\begin{equation*}
  \sum_{\substack{j=1 \\ \gcd(j,p)=1}}^{J} \frac{J-j}{J^2}
  \sum_{a=1}^{p} \left(e^{2\pi i j\frac{a^2}{p}} +e^{-2\pi i
    j\frac{a^2}{p}} \right)
\end{equation*}
apply first Abel's identity (Lemma \ref{Abel})
\begin{equation*}
  = \int_{1}^{J} \frac{1}{J^2}\sum_{\substack{j=1\\gcd(j,p)=1}}^{u}
  \sum_{a=1}^{p} \left( e^{2\pi i j\frac{a^2}{p}} +e^{-2\pi i
    j\frac{a^2}{p}} \right) \mathrm{d}u
\end{equation*}
then apply Lemma \ref{huanglemma1}
\begin{equation*}
  \leq \frac{2}{J^2}\sqrt{p} \int_{1}^{J} \sum_{\substack{j=1
      \\ \gcd(j,p)=1}}^{u} \left( \frac{j}{p}\right) \mathrm{d} u.
\end{equation*}
To estimate the sum of Dirichlet characters, we apply Lemma
\ref{lemmaMcGown} as well as Lemma \ref{logbound} for $\varepsilon =
0.249$ (Since $ \frac{10.0366}{(e 0.249)^{\frac14}} \geq
3.41105405758246$, we use the bigger constant in both the $p \leq
2\cdot 10^4$ and the $2\cdot 10^4 < p$ case):
\begin{equation*}
  \leq \frac{4\cdot 10.0366}{3 (e 0.249)^{\frac14}} J^{-\frac12}
  p^{\frac{11}{16} + \frac{0.249}{4}}.
\end{equation*}
Recall $J=\left\lfloor \frac{1}{6\kappa\psi(p)}\right\rfloor \neq 0$,
so $J^{-\frac12} \leq \sqrt{12 \kappa \psi(q)}$ and $\psi(p) \geq
p^{-\frac{5}{8}+\eta}$.  In total, we get
\begin{align*}
  \lambda\left(\mathcal{B}'(\kappa,\psi)^c \right) &\leq 2
  \kappa\sum_{p \text{ prime}} \frac{\psi(p)}{p} \sum_{\substack{a=1
      \\ \left\lVert \frac{a^2}{p} \right\rVert \leq 3\kappa\psi(p)}}
  1 \\ &\leq \kappa \frac{\pi^2}{2}\left( \frac{4\sqrt{12}\cdot
    10.0366}{3(e\cdot 0.249)^{\frac14}} \sqrt{\kappa}\sum_{p\text{
      prime}} \psi(p)^{\frac32} p^{-\frac{5}{16}+\frac{0.249}{4}} +
  \sum_{\substack{p\text{ prime } \\ p\leq \left\lfloor
      \frac{1}{6\kappa\psi(p)}\right\rfloor}} \frac{\psi(p)}{p} +
  12\kappa \sum_{\substack{p\text{ prime } \\ p > \left\lfloor
      \frac{1}{6\kappa\psi(p)}\right\rfloor}} \psi(p)^2 \right)
  \\ &\leq \kappa^{\frac32} \frac{\pi^2 2\sqrt{12}\cdot
    10.0366}{3(e\cdot 0.249)^{\frac14}} S_{\psi} + \kappa
  \frac{\pi^2}{2} S_{\psi}+ \kappa^2 6\pi^2 S_{\psi},
\end{align*}
which is $<\delta$ for
\begin{align*}
  \kappa(\delta) &< \min\left(\left(\frac{\delta (e\cdot
    0.249)^{\frac14}}{2 \pi^2\sqrt{12}\cdot 10.0366
    S_{\psi}}\right)^{\frac23}, \frac{2\delta}{3\pi^2 S_{\psi}},
  \sqrt{\frac{\delta}{18\pi^2 S_{\psi}}}\right)\\ &\simeq \min\left(
  \frac{\delta^{\frac23}}{S_{\psi}^{\frac23}} 0.0120432671,
  \frac{\delta}{S_{\psi}} 0.0675474558, \sqrt{\frac{\delta}{S_{\psi}}}
  0.0750263597 \right).
\end{align*}
\hfill $\square$

\section{Concluding remarks}

The results of this paper give quantitative Khintchine theorems for
convergence under four different conditions on the approximating
rationals. The varying $\kappa$ is very small indeed for the most
general of the results. The problem is in the explicit constants
arising from the various variants of Burgess' theorem applied, but
more specifically from the occurrence of the divisor function in these
settings. It occurs twice, once in the theorem for primitive
characters and once in the passage from primitive to general
characters. While a better estimate is desirable, it is in our opinion
unlikely to be attainable with the methods of this paper. 

\appendix %--------------------------------------APPENDIX--------------------------------------
\section{Python code}\label{SectionPython} %--------PYTHON CODE--------

\begin{python}
from sympy import kronecker_symbol
import math

FILENAME = "BurgessKonstantLogV2.txt"

# Read the last saved values from the log file so we can pick up where we left off:
with open(FILENAME, "r") as file:
last = file.readlines()[-2:]
qstart = int(last[0])
maks = float(last[1])

# Runs until manually interrupted with Ctrl+C
while True:  
	for q in range(qstart, qstart + 500):
		lower = math.floor(q ** (3 / 8))
		upper = math.ceil(q ** (5 / 8))
		q_factor = q ** (3 / 16)  # Constant for this q, so computed once outside the inner loop
		
		# Build the running sum up to lower:
		running_sum = sum(kronecker_symbol(N, q) for N in range(1, lower))
		
		# Check all N values in the relevant range [lower, upper]
		for N in range(lower, upper):
			running_sum += kronecker_symbol(N, q)  # Extend the sum by one term rather than recomputing from scratch
			VS = abs(running_sum)
			HS = math.sqrt(N) * q_factor
			maks = max(maks, VS / HS)  # Update the maximum if we found a larger value
	
	# After completing a batch of 500 q values, save progress to the log file:
	qstart += 500
	with open(FILENAME, "a") as file:
	file.write("\n" + str(qstart))  # Save where to start next time
	file.write("\n" + str(maks))  # Save the largest value found so far
\end{python}
Running this results in Lemma \ref{burgesssmallq} as well as the constant mentioned after Lemma \ref{lemmaMcGown}. Included is a snippet of the txt-file, \texttt{BurgessKonstantLogV2.txt}, the program produces (The first two lines containing zeroes were added manually before running the program).
\begin{python}
	0
	0
	500
	1.96397891158409
	1000
	2.29264220735264
	1500
	2.40876073940135
	2000
	2.51290888600514
	2500
	2.57511197155484
	3000
	2.66040431048477
	3500
	2.73314435040598
	4000
	2.75794837745823
	4500
	2.82596689219766
	5000
	2.82596689219766
	...
	19000
	3.39672051063381
	19500
	3.41105405758246
	20000
	3.41105405758246
	20500
	3.41105405758246
	21000
	3.41105405758246
	...
	99500
	4.19253564643679
	100000
	4.19253564643679
	100500
	4.20598018618317
	101000
	4.20598018618317
\end{python}

\bibliographystyle{siam} 
\bibliography{refskhintchine.bib}

\end{document}